\documentclass[10pt]{amsart}
\usepackage{graphicx}
\usepackage{amscd}
\usepackage{amsmath}
\usepackage{amsthm}
\usepackage{amsfonts}
\usepackage{amssymb}
\usepackage{mathrsfs}
\usepackage{bm}
\usepackage{enumerate}
\usepackage{amsrefs}
\usepackage{xcolor}
\usepackage[colorlinks, citecolor=blue, linkcolor=red, pdfstartview=FitB]{hyperref}
\usepackage[latin1]{inputenc}

\usepackage{tikz}
\usepackage{tikz-cd}
\usepackage{caption}
\usetikzlibrary{decorations.markings}
\usepackage{lipsum}
\usepackage{mathrsfs}

\usepackage{marginnote}	
\usepackage{soul} 			

\newcommand{\bC}{{\mathbb{C}}}

\newcommand{\bF}{{\mathbb{F}}}

\newcommand{\bM}{{\mathbb{M}}}

  \newcommand{\E}{{\mathcal{E}}}
  \newcommand{\F}{{\mathcal{F}}}

\renewcommand{\O}{{\mathcal{O}}}

\renewcommand{\S}{{\mathcal{S}}}

\newcommand{\rC}{\mathrm{C}}

\renewcommand{\phi}{\varphi}

\newcommand{\upchi}{{\raise.35ex\hbox{$\chi$}}}

\newcommand{\mycomment}[1]{}

\newcommand{\Ab}{\operatorname{Ab}}

\newcommand{\CE}{\operatorname{CE}}

\newcommand{\Exp}{\operatorname{Ex}}

\newcommand{\id}{\operatorname{id}}

\newcommand{\ran}{\operatorname{ran}}

\newcommand{\spn}{\operatorname{span}}

\newtheorem{lemma}{Lemma}[section]
\newtheorem{theorem}[lemma]{Theorem}
\newtheorem{proposition}[lemma]{Proposition}
\newtheorem{corollary}[lemma]{Corollary}
\newtheorem{theoremx}{Theorem}

\theoremstyle{definition}

\newtheorem{example}{Example}

\date{\today}

\author{Rapha\"el Clou\^atre}
\address{Department of Mathematics, University of Manitoba, Winnipeg, Manitoba, Canada R3T 2N2}
\email{raphael.clouatre@umanitoba.ca\vspace{-2ex}}

\author{Colin Krisko}
\email{kriskoc4@myumanitoba.ca\vspace{-2ex}}

\thanks{R.C. was partially supported by an NSERC Discovery Grant.}

\title[$\rC^*$-supports and abnormalities]{$\rC^*$-supports and abnormalities of operator systems}
\begin{document}
\begin{abstract}
Let $S$ be a concrete operator system represented on some Hilbert space\ $H$. A $\rC^*$-support of $S$ is the $\rC^*$-algebra generated (via the Choi--Effros product) by $S$ inside an injective operator system  acting on $H$. By leveraging Hamana's theory, we show that such a $\rC^*$-support is unique precisely when $\rC^*(S)$ is contained in every copy of the injective envelope of $S$ that acts on $H$. Further, we demonstrate how the uniqueness of certain $\rC^*$-supports can be used to give new characterizations of the unique extension property for $*$-representations, as well as the hyperrigidity of $S$. In another direction, we utilize the collection of all $\rC^*$-supports of $S$ to describe the subspace generated by the so-called abnormalities of $S$, thereby complementing a result of Kakariadis.
\end{abstract}
\maketitle

\section{Introduction}

In the theory of operator systems, a fundamental problem is to identify certain special extreme points among completely positive maps; these are the so-called absolute extreme points \cite{EHKM2018},\cite{EH2019} or the nc extreme points \cite{DK2019}. 
To be more precise, fix a Hilbert space $H$ and let $S\subset B(H)$ be an operator system. Put $A=\rC^*(S)$. Let  $\pi:A\to B(H_\pi)$ be a unital $*$-representation. Let $E$ denote the set of unital completely positive maps $\phi:A\to B(H_\pi)$ agreeing with $\pi$ on $S$.  When $E$ reduces to the singleton $\{\pi\}$, we say that $\pi$ has \emph{unique extension property} with respect to $S$. As is now well understood, the irreducible $*$-representations of this type  can be meaningfully interpreted as points in the noncommutative Choquet boundary of $S$ \cite{kleski2014bdry},\cite{DK2019},\cite{CTh2023}. Furthermore, the question of $E$ being a singleton relates to computations of the $\rC^*$-envelope of $S$ \cite{dritschel2005},\cite{arveson2008},\cite{DK2015}, as well as the hyperrigidity conjecture (see \cite{DK2021},\cite{CS2023},\cite{BDO2024},\cite{CTh2024}, \cite{bilich2025} and the references therein). 

Since the unique extension property is well behaved with respect to direct sums, in determining whether $\pi$ has the unique extension property, we may assume without loss of generality that $\pi$ is in fact injective. To simplify notation in the foregoing discussion, we thus choose to focus on the case where $\pi$ is simply  the identity representation of $A$. When $A$ contains the ideal of compact operators on $H$, Arveson's famous boundary theorem \cite[Theorem 2.1.1]{arveson1972} gives a simple mechanism to detect the unique extension property. Our main results do not require this assumption, and are altogether quite different in spirit, as we explain next.

In order for $E$ to be a singleton,  it must be that $\phi(a)=a$ for every $a\in A$ and $\phi\in E$. Following the terminology introduced by Kakariadis in \cite{kakariadis2013}, an element of the form  $\phi(a)-a$, for $\phi\in E$ and $a\in A$, will be called an \emph{abnormality} of $S$. Throughout, we denote by $\Ab(S)\subset B(H)$ the collection of all abnormalities.  
 Our new contributions in this paper are inspired by a result from  \cite{kakariadis2013} about the  ``restricted" abnormalities $\Ab(S)\cap A$, which utilized the theory of injective envelopes. To properly frame our work, we recount some of the details below; a more thorough account is given in Section \ref{S:injenv}.

While the original impetus to study the unique extension property was Arveson's proposed program for constructing the $\rC^*$-envelope \cite{arveson1969}, this took several decades to fully crystallise. In the meantime, Hamana \cite{hamana1979} had given the first proof of existence of the $\rC^*$-envelope, building on a rather different set of tools and perspective.  An \emph{$S$-projection} is a unital completely positive idempotent $\theta:B(H)\to B(H)$ such that $\theta|_S=\id$. There is a natural partial order defined on the set of all $S$-projections, and it can be shown that minimal elements exist. These minimal $S$-projections are pivotal in Hamana's construction. Indeed, if $\theta$ is a minimal $S$-projection, then its range $R$ is a so-called \emph{injective envelope} for $S$ that can, in addition, be turned into a $\rC^*$-algebra using the Choi--Effros product induced by $\theta$. Furthermore, the $\rC^*$-subalgebra generated by $S$ inside of $R$ can be verified to be the $\rC^*$-envelope of $S$. 

Naturally, then,  minimal $S$-projections carry significant information about the operator system $S$. Kakariadis further reinforced this by showing in \cite{kakariadis2013} that
\begin{equation}\label{Eq:Kak}
\Ab(S)\cap A=\ker \theta \cap A
\end{equation}
 for any minimal $S$-projection $\theta$. As well, he showed that this set coincides with the so-called Shilov ideal of $S$ inside of $A$, as introduced by Arveson in \cite{arveson1969}. Notably,  not only does this statement describe a portion of $\Ab(S)$, but it also shows that $\ker \theta\cap A$ is \emph{independent} of $\theta$. (While the results of \cite{kakariadis2013} are stated within the framework of unital completely contractive maps on unital operator spaces, there is a well-known equivalence between this category and our category of interest consisting of unital completely positive maps on operator systems; see \cite[Proposition 2.12]{paulsen2002}).

 From the point of view of the foregoing discussion on abnormalities and their connection with the unique extension property for the identity representation of $A$, it appears as though \eqref{Eq:Kak} paints an incomplete picture. Indeed, the natural problem that arises is to describe the full set $\Ab(S)$, and not just its intersection with $A$. In Section \ref{S:Cstarext}, we aim to understand the configuration of all injective envelopes of $S$ inside of $B(H)$, relative to $A$.  While our motivation is to use the resulting information to describe the full set of abnormalities in order to improve \eqref{Eq:Kak}, we study this question rather extensively for its own sake and obtain results of independent interest. It should be mentioned that this approach of leveraging not just a single copy, but rather all copies, of the injective envelope, has been employed previously towards different goals \cite{paulsen2011}. In order to summarize our results, we introduce some terminology. 

A \emph{$\rC^*$-extension} of $S$ is a pair $(B,j)$ consisting of a unital $\rC^*$-algebra $B$ along with a unital completely isometric map $j:S\to B$ such that $B=\rC^*(j(S))$.  An operator system $X\subset B(H)$ containing $S$ will be said to be a \emph{$\rC^*$-support of $S$} if there is a $\rC^*$-extension $(B,j)$ of $S$ along with a unital surjective completely isometric map  $\Gamma:B\to X$ such $\Gamma\circ j=\id$. In other words, a $\rC^*$-support is, roughly speaking, a linear copy of a $\rC^*$-extension. We note here that operator systems that are completely isometric to $\rC^*$-algebras have been of interest in various settings; see for instance \cite{KW1998},\cite{CW2025},\cite{courtney2025} and the references therein.
 
We show in Theorem \ref{T:Cstarext} that all $\rC^*$-supports arise in a rather concrete way, namely from the Choi--Effros product induced by an $S$-projection. Furthermore, this projection may be chosen to be minimal precisely when $(B,j)$ is a $\rC^*$-envelope of $S$ (Corollary \ref{C:coversenv}).
The following is a summary of Theorem  \ref{T:Shtrivial} along with Corollary \ref{C:uniqueCE}, and is our first main result.

\begin{theoremx}\label{T:A}
Let $S\subset B(H)$  be an operator system and let $A=\rC^*(S)$. Consider the following statements.
\begin{enumerate}[{\rm (i)}]
\item $A$ is the unique $\rC^*$-support of $S$.
\item $A$ is contained  in every copy of the injective envelope of $S$ inside $B(H)$.
\item $A$ is a $\rC^*$-envelope of $S$.
\item $A$ is contained in some copy of the injective envelope of $S$ inside $B(H)$.
\end{enumerate}
Then, we have 
\[
{\rm (i)}\Leftrightarrow{\rm (ii)}\Rightarrow{\rm (iii)}\Leftrightarrow{\rm (iv)}.
\]
\end{theoremx}
The proof of this result requires a fairly delicate analysis of Choi--Effros products, which we perform in Section \ref{S:injenv}. As an application, we combine Theorem \ref{T:A} with results of Paulsen \cite{paulsen2011} to give  in Corollary  \ref{C:uep} the following new characterization of the unique extension property for injective $*$-representations.

\begin{theoremx}\label{T:C}
Let $\pi:A\to B(H)$ be an injective unital $*$-representation.  Then, the following statements are equivalent.
\begin{enumerate}[{\rm (i)}]
\item $\pi$ has the unique extension property with respect to $S$. 
\item $\pi(A)$ is the unique $\rC^*$-support of $\pi(S)$.
\end{enumerate}
\end{theoremx}

Similarly, Corollary \ref{C:hr} characterizes  hyperrigid operator systems using our notion of $\rC^*$-support. 

In Section \ref{S:ab}, we turn to the question of describing the full set of abnormalities. 
We define the \emph{$\rC^*$-expanse} of $S$ as
\[
\rC^*\Exp(S)=\spn\{\theta(A) \}
\] 
where  the span is taken over all  minimal $S$-projections $\theta$. We show in Theorem \ref{T:expanse} that, when the Shilov ideal of $S$ in $A$ is trivial, then the $\rC^*$-expanse of $S$ coincides with the subspace generated by all $\rC^*$-supports of $S$. 
Our second main result reads as follows; see Theorem \ref{T:abdescr}.

\begin{theoremx}\label{T:B}
Assume that $A$ is a $\rC^*$-envelope of $S$. Then,
we have  \[\spn \Ab(S)=\rC^*\Exp(S)\cap \ker \theta\] for every minimal $S$-projection $\theta$.
\end{theoremx}

Much like in Kakariadis' result \eqref{Eq:Kak}, we find that the intersection $\rC^*\Exp(S)\cap \ker\theta$ does not depend on the choice of minimal projection $\theta$.

It is a consequence of Hamana's theory  that, for a minimal $S$-projection $\theta$, the quotient map $B(H)\to B(H)/\ker \theta$  enjoys what we call the \emph{unique $(B(H)/\ker \theta)$-extension property} with respect to $S$ (see Proposition \ref{P:quotienttight}). Here, $B(H)/\ker \theta$ is given the operator system quotient structure.
This naturally leads us to the question of whether there is a quotient of $B(H)$ which, similarly, ``forces" some sort of unique extension property on the identity representation of $A$.
More precisely, we are looking for a closed subspace $D\subset B(H)$, with corresponding quotient map $\delta:B(H)\to B(H)/D$, for which $\delta\circ \psi=\delta$ whenever  $\psi:A\to B(H)$ is a unital completely positive map acting as the identity on $S$.  To avoid trivialities, we also require that $\delta$ be completely isometric on $S$, where $B(H)/D$ is given the operator space quotient structure. As an application of Theorem \ref{T:B}, we show in Corollary \ref{C:uepquotient} that the norm-closed subspace generated by $\Ab(S)$ has this property.

Finally, in Section \ref{S:example}, we close the paper by analysing the set of abnormalities in some concrete examples, and by contrasting it with Kakariadis' space $\Ab(S)\cap A$.

\section{Injective envelopes and Choi--Effros products}\label{S:injenv}

In this preliminary section, we gather some background material and prove some technical facts that are required later.

\subsection{Hamana's theory of injective envelopes}

We will use standard terminology and facts for operator spaces and operator systems, as can be found in many references such as \cite{ER2022} and \cite{paulsen2002}.

Let $R$ be an operator system. Recall that $R$ is said to be \emph{injective} if, given an inclusion $S\subset T$ of operator systems and a unital completely positive map $\psi:S\to R$, there exists a unital completely positive map $\Psi:T\to R$ extending $\psi$. In the case where $R$ is concretely represented on some Hilbert space $H$ (so that $R\subset B(H)$), it easily follows from Arveson's extension theorem that $R$ is injective precisely when there is a unital completely positive idempotent map $\theta:B(H)\to B(H)$ with range equal to $R$.

Next, let $S$ be an arbitrary operator system. An \emph{extension} of $S$ is a pair $(R,\iota)$ consisting of an operator system $R$ and a unital complete isometry $\iota:S\to R$. When $R$ happens to be injective, then we say that $(R,\iota)$ is an injective extension of $S$.
An injective extension $(R,\iota)$ will be said to be an \emph{injective envelope} of $S$ if it has the property that given any other injective operator system $T$ with $\iota(S)\subset T\subset R$, we must have $T=R$. 

There are two other important properties that are relevant for our purposes. Let $(R,\iota)$ be an extension of $S$. We say that $(R,\iota)$ is \emph{rigid} if the only unital completely positive map $\phi:R\to R$ satisfying $\phi\circ \iota=\iota$ is the identity map. Also, we say that $(R,\iota)$ is \emph{essential} if, whenever $X$ is an operator space and  $\phi:R\to X$ is a completely contractive map such that $\phi\circ\iota$ is completely isometric, it follows that in fact $\phi$ is completely isometric.

The following summarizes the well-known relationships between these notions, which we record here for ease of reference below.

\begin{proposition}\label{P:injenvchar}
Let $S$ be an arbitrary operator system and let $(R,\iota)$ be an extension of $S$. Then, the following statements are equivalent.
\begin{enumerate}[{\rm (i)}]
\item $(R,\iota)$ is an injective envelope of $S$.
\item $R$ is injective and $(R,\iota)$ is a rigid extension of $S$.
\item $R$ is injective and $(R,\iota)$ is an essential extension of $S$.
\end{enumerate}
\end{proposition}
\begin{proof}
(i)$\Rightarrow$(ii): See \cite[Corollary 15.7]{paulsen2002}.

(ii)$\Leftrightarrow$(iii): See \cite[Lemma 3.7]{hamana1979}.

(ii)$\Rightarrow$(i): Let $T$ be an injective operator system with $\iota(S)\subset T\subset R$. Let $\phi:T\to R$ denote the inclusion map, which is unital, completely isometric and satisfies $\phi\circ\iota=\iota$. Let $\psi:R\to T$ be a unital completely positive extension of the identity $T\to T$. Then, $\phi\circ \psi:R\to R$ satisfies $\phi\circ \psi\circ \iota=\iota$. Since $(R,\iota)$ is rigid, we infer that $\phi\circ\psi$ is the identity on $R$. Thus, $\phi$ is surjective and $R=T$.
\end{proof}

A straightforward consequence of this result is the uniqueness of the injective envelope. More precisely, if  $(R,\iota)$ and $(E,j)$ are two injective envelopes of $S$, then  there is a unital surjective completely isometric map $\theta:R\to E$ such that $\theta\circ \iota=j$. We now turn to the question of existence of the injective envelope.

Choose a Hilbert space $H$ on which $S$ is concretely represented, so that we may view $S$ as an operator system in $B(H)$. A unital completely positive map $\theta:B(H)\to B(H)$ is called an \emph{$S$-map} if $\theta|_S=\id$. If, in addition, $\theta$ is idempotent (i.e.  $\theta\circ \theta=\theta$), then $\theta$ is called an \emph{$S$-projection}. We shall  make use of the standard  partial order on the set of all $S$-projections, defined as follows: $\psi\prec\theta$ precisely when $\psi\circ \theta=\theta\circ \psi=\psi$. The following is one of Hamana's fundamental ideas \cite{hamana1979}.

\begin{theorem}\label{T:Hamana}
Let $S\subset B(H)$ be an operator system and let $i:S\to B(H)$ denote the inclusion. Let $\theta:B(H)\to B(H)$ be an $S$-projection. Then, $\theta$ is minimal if and only if $( \ran \theta, i)$ is an injective envelope of $S$.   
\end{theorem}
\begin{proof}
This follows from Proposition \ref{P:injenvchar} along with \cite[Theorem 2.9]{paulsen2011}.
\end{proof}

In light of this, we see that the existence of injective envelopes reduces to the existence of minimal projections. The following result insures that these always exist.

\begin{theorem}\label{T:minproj}
Let $S\subset B(H)$ be an operator system.  Let $\psi:B(H)\to B(H)$ be an $S$-projection. Then, there is a minimal $S$-projection $\theta:B(H)\to B(H)$ with $\ran\theta\subset \ran \psi$.
\end{theorem}
\begin{proof}
By \cite[Theorem 3.5]{hamana1979}, there is a minimal $S$-projection $\theta_0:B(H)\to B(H)$. If we put $\theta=\psi\circ \theta_0$, then by \cite[Proposition 2.11]{paulsen2011} it follows that $\theta$ is a minimal $S$-projection as well, and by construction its range is contained in the range of $\psi$.
\end{proof}

\subsection{Choi--Effros products and the $\rC^*$-envelope}
Fix a Hilbert space $H$ and let $\theta:B(H)\to B(H)$ be a unital completely positive idempotent map. Let $R\subset B(H)$ denote the range of $\theta$, which is then an injective operator system. The \emph{Choi--Effros product} $*_\theta:R\times R\to R$  is defined as
\[
s *_\theta t=\theta(st), \quad s,t\in R.
\]
Equipped with this product, $R$ becomes a unital $\rC^*$-algebra that we  denote by  $\CE(R,\theta)$. The inclusion map $\alpha_\theta: R\to \CE(R,\theta)$ is unital, surjective and completely isometric \cite[Theorem 3.1]{CE1977}. Throughout our work, it will be important to distinguish between $R$ and $\CE(R,\theta)$, which will be accomplished by keeping track of the map $\alpha_\theta$. Given an operator system  $X\subset R$, we denote by $\CE(X,\theta)$ the $\rC^*$-subalgebra generated by $\alpha_\theta(X)$ inside of $\CE(R,\theta)$.

Since $\theta$ is idempotent and hence acts as the identity on its range $R$, we see that 
\[
\theta(st)=\theta(\theta(s)\theta(t))=\theta(s)*_\theta \theta(t), \quad s,t\in R.
\]
Thus,  $\alpha_\theta\circ \theta:\rC^*(R)\to \CE(R,\theta)$ is a surjective $*$-homomorphism. In fact, much more is true as we shall see below in Proposition \ref{P:minuep}.

Next, let $S\subset B(H)$ be an arbitrary operator system. Recall from the introduction that a \emph{$\rC^*$-extension} of $S$ is a pair $(A,j)$ consisting of a unital $\rC^*$-algebra $A$ along with a unital completely isometric map $j:S\to A$ such that $A=\rC^*(j(S))$. A $\rC^*$-extension $(A,j)$ is a \emph{$\rC^*$-envelope} of $S$ if it satisfies the following universal property: given another $\rC^*$-extension $(B,\iota)$, there must exist a surjective unital $*$-homomorphism $\pi:B\to A$ such that $\pi\circ \iota=j$.  We then say that $\ker\pi$ is the \emph{Shilov ideal} of $\iota(S)$ in $B$. It is easily seen that a $\rC^*$-envelope is essentially unique, and that the Shilov ideal is unique.
On the other hand, the existence question for these objects is highly non-trivial, and it was first resolved by Hamana  \cite[Theorem 4.1]{hamana1979}, as follows.

\begin{theorem}\label{T:Cstarenv}
Let $S\subset B(H)$ be an operator system. Let $\theta:B(H)\to B(H)$ be a minimal $S$-projection. Then, $(\CE(S,\theta),\alpha_\theta|_S)$ is a $\rC^*$-envelope of $S$.
\end{theorem}

In the remainder of this section we collect some further technical facts about Choi--Effros products and their corresponding $\rC^*$-algebras that will be of use later.
We start with a basic observation, found in \cite[Theorem 4.1]{CE1976}.

\begin{proposition}\label{P:CEtrik}
Let $A$ and $B$ be unital $\rC^*$-algebras and let $\gamma:A\to B$ be a unital completely isometric map. Then, the following statements hold.
\begin{enumerate}[{\rm (i)}]
\item There is a unital $*$-homomorphism $\eta:\rC^*(\gamma(A))\to A$ such that $\eta\circ \gamma=\id$.
\item If  $\psi:\rC^*(\gamma(A))\to A$ is a unital completely positive map satisfing $\psi\circ \gamma=\id$ on $A$, then $\psi=\eta$. 
\end{enumerate} 
\end{proposition}

We record an important uniqueness property for minimal $S$-projections.

\begin{proposition}\label{P:minuep}
Let $S\subset B(H)$ be an operator system and let $A=\rC^*(S)$.  Let $\theta:B(H)\to B(H)$ be a minimal $S$-projection with range $R$. Let $\pi:B(H)\to R$ be an $S$-map. Then, $\alpha_\theta\circ \pi|_{A}:A\to \CE(R,\theta)$ is a $*$-homomorphism. Moreover, $\pi$ agrees with $\theta$ on $R+A$.
\end{proposition}
\begin{proof}
Let $i:S\to B(H)$ denote the inclusion. By Proposition \ref{P:injenvchar} and Theorem \ref{T:Hamana}, $(R,i )$ is a rigid extension of $S$. Since $\pi|_S=\id$, we infer by rigidity that $\pi|_R=\id$. In particular, $\pi$ and $\theta$ agree on $R$.

Next, let $\psi:R\to B(H)$ be a unital completely positive extension of the identity on $S$, so that $\psi|_S=\id$. We find $\pi\circ \psi|_S=\id$, so once again by rigidity we infer $\pi\circ \psi=\id$ on $R$. In particular, $\psi$ is completely isometric and $\alpha_\theta\circ \pi\circ \psi\circ \alpha_\theta^{-1}=\id$ on $\CE(R,\theta)$. By applying Proposition \ref{P:CEtrik} with $\gamma=\psi\circ\alpha_\theta^{-1}:\CE(R,\theta)\to B(H)$, we find that $\alpha_\theta\circ \pi$ restricts to a $*$-homomorphism on the $\rC^*$-subalgebra of $B(H)$ generated by $\psi(R)$, which in particular contains $A$.  

Applying the argument of the previous paragraph with $\theta$ instead of $\pi$,  we see that $\alpha_\theta\circ \theta|_A $ is a $*$-homomorphism. Finally, since  $\alpha_\theta\circ \theta|_A $ and $\alpha_\theta\circ \pi|_A$ are $*$-homomorphisms agreeing on $S$, they must agree everywhere. Using that $\alpha_\theta$ is completely isometric on $R$,  it follows that $\theta$ and $\pi$ agree on $A$ as well.
\end{proof}

Next, we establish a crucial technical tool, which allows for the comparison of various Choi--Effros products.

\begin{proposition}\label{P:CEcomp}
Let $S\subset B(H)$ be an operator system.  Let $\psi:B(H)\to B(H)$ be an $S$-projection and let $\theta:B(H)\to B(H)$ be a minimal $S$-projection. Then, the following statements hold.
\begin{enumerate}[{\rm(i)}]
\item The map \[\pi=\alpha_\theta\circ \theta\circ \alpha_\psi^{-1}|_{\CE(S,\psi)}: \CE(S,\psi)\to \CE(S,\theta)\] is a unital surjective $*$-homomorphism whose kernel is the Shilov ideal of $\alpha_\psi(S)$ in $\CE(S,\psi)$.
\item  If $\alpha_\psi^{-1}(\CE(S,\psi))$ is contained in the range of $\theta$, then \[\alpha_\psi^{-1}(\CE(S,\psi))=\alpha_\theta^{-1}( \CE(S,\theta)).\]
\item Assume that there is a unital $*$-homomorphism $\sigma:\CE(S,\theta)\to \CE(S,\psi)$ such that $\sigma\circ\alpha_\theta|_S=\alpha_\psi|_S$. Then, there is a minimal $S$-projection $\theta':B(H)\to B(H)$ such that $\alpha_{\theta'}^{-1}(\CE(S,\theta'))=\alpha_\psi^{-1}(\CE(S,\psi))$. 
\end{enumerate}
\end{proposition}
\begin{proof}
Throughout the proof, we let $i:S\to B(H)$ denote the inclusion. We  let $R_\psi$ and $R_\theta$ denote the ranges of $\psi$ and $\theta$ respectively, which are both injective operator systems. Also, we consider the unital $\rC^*$-algebras $A=\CE(R_\theta,\theta)$ and $B=\CE(R_\psi,\psi)$. Recall that the inclusion maps $\alpha_\theta:R_\theta \to A$ and $\alpha_\psi:R_\psi\to B$ are unital, surjective complete isometries.

(i) Since $R_\psi$ is  injective, there is a unital completely positive map $\gamma:R_\theta\to R_\psi$ extending the identity map on $S$. Then, $ \theta\circ \gamma  :R_\theta\to R_\theta$ is an $S$-map. On the other hand, $(R_\theta,i)$ is a rigid extension of $S$ by Proposition \ref{P:injenvchar} and Theorem \ref{T:Hamana}, so we find $\theta\circ \gamma=\id$ on $R_\theta$. In particular, $\gamma$ must be completely isometric on $R_\theta$, and thus so is $\alpha_\psi \circ \gamma\circ \alpha_\theta^{-1}:A\to B$. Applying Proposition \ref{P:CEtrik}(ii), we infer that the map $\alpha_\theta\circ \theta\circ \alpha_\psi^{-1}:B\to A$ restricts to a $*$-homomorphism on the $\rC^*$-subalgebra of $B$ generated by $(\alpha_\psi\circ \gamma)(R_\theta)$, which certainly contains $\alpha_\psi(S)$ and hence $\CE(S,\psi)$. 
In particular, this means that $(\alpha_\theta\circ \theta\circ \alpha_\psi^{-1})(\CE(S,\psi))$ is the $\rC^*$-subalgebra of $A$ generated by $S$, which in turn is the very definition of $\CE(S,\theta)$. Therefore 
\begin{equation}\label{Eq:CE}
(\alpha_\theta\circ \theta\circ \alpha_\psi^{-1})(\CE(S,\psi))=\CE(S,\theta).
\end{equation}
We have thus shown that the map $\pi=\alpha_\theta\circ \theta\circ \alpha_\psi^{-1}|_{\CE(S,\psi)}:\CE(S,\psi)\to A$ is a $*$-homomorphism with range $\CE(S,\theta)$.

Next, observe that $(\CE(S,\theta),\alpha_\theta|_S)$ is a $\rC^*$-envelope of $S$ by virtue of Theorem \ref{T:Cstarenv}. Since $\pi\circ \alpha_\psi|_S=\alpha_\theta|_S$ by construction, it follows that $\ker \pi$ is the Shilov ideal of $\alpha_\psi(S)$ in $\CE(S,\psi)$.

(ii) Using \eqref{Eq:CE}, we find $\alpha_\theta^{-1}(\CE(S,\theta))=(\theta\circ\alpha_\psi^{-1})(\CE(S,\psi))$. If $\alpha_\psi^{-1}(\CE(S,\psi))$ is contained in the range of $\theta$, then the fact that $\theta$ is idempotent yields
\[
\alpha_\theta^{-1}(\CE(S,\theta))=\alpha_\psi^{-1}(\CE(S,\psi))
\]
as desired.

(iii) Let $T=\alpha_\theta^{-1}(\CE(S,\theta))$, which is an operator system in $R_\theta$. By assumption, we see that $\alpha_\psi^{-1}\circ \sigma\circ \alpha_\theta:T\to R_\psi$ is an $S$-map. By injectivity of $R_\psi$, we can extend this map to an $S$-map  $\Phi:B(H)\to R_\psi$. Then, $\theta'=\Phi\circ \theta$ is another minimal $S$-projection by \cite[Proposition 2.11]{paulsen2011}. 

Recall  now that for an $S$-projection $\omega:B(H)\to B(H)$ with range $R$, $\alpha_\omega\circ \omega:\rC^*(R)\to \CE(R,\omega)$ is a $*$-homomorphism, so we infer that
$
\alpha_\omega^{-1}(\CE(S,\omega))=\omega(\rC^*(S))
$
 is the norm-closed subspace of $B(H)$ consisting of elements of the form $\omega(t_1 t_2\ldots t_n)$ for any choice of finitely many  $t_1,\ldots,t_n\in S$. 
 
 Therefore, to establish our claim that $\alpha_{\theta'}^{-1}(\CE(S,\theta'))=\alpha_\psi^{-1}(\CE(S,\psi))$, it suffices to establish that, given $t_1,\ldots,t_n\in S$ we must have $\theta'(t_1t_2\cdots t_n)=\psi(t_1t_2\cdots t_n)$. 
 To see this, note that \[(\alpha_\theta\circ \theta)(t_1t_2\ldots t_n)=t_1*_\theta t_2 *_\theta \cdots *_\theta t_n\in \CE(S,\theta)\]
so $\theta(t_1t_2\ldots t_n)\in T$ and
\begin{align*}
\theta'(t_1t_2\cdots t_n)&=\Phi( \theta(t_1t_2\ldots t_n))=(\alpha_\psi^{-1}\circ \sigma\circ \alpha_\theta \circ \theta)(t_1t_2\ldots t_n)\\
&=(\alpha_\psi^{-1}\circ \sigma)(t_1*_\theta t_2 *_\theta \cdots *_\theta t_n) \\
&=\alpha_\psi^{-1}(\sigma(t_1)*_\psi \sigma(t_2) *_\psi \cdots *_\psi \sigma(t_n))\\
&=\alpha_\psi^{-1}(t_1*_\psi t_2 *_\psi \cdots *_\psi t_n)=\psi(t_1t_2\cdots t_n)
\end{align*}
since $\sigma$ is a $*$-homomorphism. 
\end{proof}

We remark here that the special case of statement (i) where $\psi$ is the identity map is found in \cite[Corollary 3.1]{kakariadis2013}.

\section{$\rC^*$-supports and the unique  extension property}\label{S:Cstarext}

Let $S\subset B(H)$ be an operator system. In this section, we study $\rC^*$-supports as a means to examine the configuration of the $\rC^*$-algebra $\rC^*(S)$ relative to all the copies of the injective envelope of $S$ contained in $B(H)$. We then utilize the resulting information to characterize the unique extension property for injective $*$-representations.

\subsection{$\rC^*$-supports and containment in an injective envelope}

An operator system $X\subset B(H)$ containing $S$ is a \emph{$\rC^*$-support of $S$} if there is a $\rC^*$-extension $(B,j)$ of $S$ along with a unital, surjective, completely isometric map  $\Gamma:B\to X$ such $\Gamma\circ j=\id$. We note that if $X$ is a $\rC^*$-support of $S$, then it is a so-called \emph{$\rC^*$-system} in the sense of \cite{KW1998}. 

The collection of all $\rC^*$-supports of $S$ is actually exhausted by Choi--Effros products, as we show next.

\begin{theorem}\label{T:Cstarext}
Let $S\subset B(H)$  be an operator system. Let $X\subset B(H)$ be an operator system containing $S$. Then, the following statements are equivalent.
\begin{enumerate}[{\rm (i)}]
\item $X$ is a $\rC^*$-support of $S$.
\item There is an $S$-projection $\theta:B(H)\to B(H)$ such that  $X=\alpha_\theta^{-1}(\CE(S,\theta))$.
\end{enumerate}
\end{theorem} 
\begin{proof}

(ii) $\Rightarrow$ (i):  Let $B=\CE(S,\theta)$. Then, $(B,\alpha_\theta|_S)$ is a $\rC^*$-extension of $S$. If we let $\Gamma=\alpha_\theta^{-1}:B\to X$, then $\Gamma$ is  unital, surjective, completely isometric map such that $\Gamma \circ \alpha_\theta=\id$, as needed.

(i) $\Rightarrow$ (ii):  By assumption, there is a $\rC^*$-extension $(B,j)$ of $S$ along with a unital, surjective, completely isometric map $\Gamma:B\to X$ such $\Gamma\circ j=\id$. Choose an isometric unital $*$-representation $\sigma:B\to B(K)$. Apply Theorems \ref{T:Hamana} and \ref{T:minproj} to the operator system $X$ to find a minimal $X$-projection $\theta:B(H)\to B(H)$ whose range $R$ is an injective envelope of $X$.  Then,  $\alpha_\theta\circ \Gamma:B\to \CE(R,\theta)$ is unital and completely isometric, and $\rC^*((\alpha_\theta\circ \Gamma)(B))=\CE(X,\theta)$. By Proposition \ref{P:CEtrik}, there is a unital $*$-homomorphism $\pi:\CE(X,\theta)\to B$ such that $\pi\circ\alpha_\theta\circ \Gamma=\id$ on $B$. In particular, this implies that $\sigma\circ \pi\circ \alpha_\theta|_X=\sigma\circ \Gamma^{-1}$ is completely isometric. 

Let $\psi:R\to B(K)$ be a unital completely positive extension of $\sigma\circ \pi\circ\alpha_\theta:\alpha_\theta^{-1}(\CE(X,\theta))\to B(K)$. Then, $\psi$ is completely isometric on $X$ by the previous paragraph. Letting $i:X\to B(H)$ denote the inclusion, $(R,i)$ is an essential extension of $X$ by Proposition \ref{P:injenvchar}. Consequently, we infer that $\psi$ is completely isometric on $R$ and hence  on $\alpha_\theta^{-1}(\CE(X,\theta))$. In other words, $\sigma\circ \pi\circ\alpha_\theta$ is completely isometric on $ \alpha_\theta^{-1}(\CE(X,\theta))$. In turn, because $\sigma$ is isometric, this implies that $\pi$ must also be  isometric, so it is a $*$-isomorphism. 

Finally, observe that, on one hand, 
\[
\pi(\CE(S,\theta))=\rC^*((\pi \circ \alpha_\theta)(S))=\rC^*((\pi\circ\alpha_\theta\circ \Gamma \circ j)(S))=\rC^*(j(S))=B
\]
and, on the other,
\[
\pi(\alpha_\theta(X))=(\pi\circ \alpha_\theta\circ  \Gamma)(B)=B.
\]
Since $\pi$ is a $*$-isomorphism, we must have $X=\alpha_\theta^{-1}(\CE(S,\theta))$.
\end{proof}

In light of the previous result, it is natural to wonder what kind of $\rC^*$-support of $S$ arises via the Choi--Effros product associated to a \emph{minimal} $S$-projection. Theorem \ref{T:Cstarenv} indicates that an answer should involve the $\rC^*$-envelope in some way. More precisely, we have the following.

\begin{corollary}\label{C:coversenv}
Let $S\subset B(H)$  be an operator system. Let $X\subset B(H)$ be an operator system containing $S$.  Then, the following statements are equivalent.
\begin{enumerate}[{\rm (i)}]
\item If $(A,j)$ is a $\rC^*$-envelope of $S$, then there is a unital, surjective, completely isometric map  $\Gamma:A\to X$ such $\Gamma\circ j=\id$. 
\item There is a minimal $S$-projection $\theta:B(H)\to B(H)$ such that  \[X=\alpha_\theta ^{-1}(\CE(S,\theta)).\]
\end{enumerate}
\end{corollary} 
\begin{proof}
(ii) $\Rightarrow$ (i): By the uniqueness property of the $\rC^*$-envelope, it suffices to establish the desired statement for a specific choice of $\rC^*$-envelope of $S$. Thus, by virtue of Theorem \ref{T:Cstarenv}, we may choose $(A,j)=(\CE(S,\theta),\alpha_\theta|_S)$, in which case it suffices to take $\Gamma=\alpha_\theta^{-1}$ .

(i) $\Rightarrow$ (ii): By Theorems \ref{T:minproj} and \ref{T:Cstarenv}, there is a minimal $S$-projection $\omega:B(H)\to B(H)$ such that $(\CE(S,\omega),\alpha_{\omega}|_S)$ is a $\rC^*$-envelope of $S$. By assumption, we obtain a unital, surjective, completely isometric map $\Gamma:\CE(S,\omega)\to X$ such that $\Gamma\circ \alpha_{\omega}|_S=\id$.

Next, note that our assumption implies in particular that  $X$ is a $\rC^*$-support of $S$. By applying Theorem \ref{T:Cstarext} we can find an $S$-projection $\psi:B(H)\to B(H)$ such that $X=\alpha_\psi^{-1}(\CE(S,\psi))$.  Then, the map $\pi=\alpha_\psi\circ\Gamma:\CE(S,\omega)\to \CE(S,\psi)$ is a unital, surjective, completely isometric map between two $\rC^*$-algebras, so by Proposition \ref{P:CEtrik} it must be a $*$-isomorphism.
Moreover, we have
\[
\pi\circ \alpha_\omega|_S=\alpha_\psi\circ \Gamma\circ \alpha_\omega|_S=\alpha_{\psi}|_S.
\]
Applying Proposition \ref{P:CEcomp}(iii), we may find another minimal $S$-projection $\theta:B(H)\to B(H)$ such that 
\[
\alpha_{\theta}^{-1}(\CE(S,\theta))=\alpha_\psi^{-1}(\CE(S,\psi))=X.
\]
\end{proof}

We now combine our findings to elucidate the relative configuration of $\rC^*(S)$ with respect to the various injective envelopes of $S$ inside of $B(H)$. The following is the first main result of the paper.

\begin{theorem}\label{T:Shtrivial}
Let $S\subset B(H)$  be an operator system. Then, the following statements are equivalent.
\begin{enumerate}[{\rm (i)}]
\item If $X\subset B(H)$ is a $\rC^*$-support of $S$ and $(A,j)$ is a $\rC^*$-envelope of $S$, then there is a unital, surjective, completely isometric map $\Gamma:A\to X$ such that $\Gamma\circ j=\id$. 
\item $(\rC^*(S),i)$ is a  $\rC^*$-envelope of $S$, where $i:S\to B(H)$ is the inclusion map.
\item $\rC^*(S)$ is contained in some injective operator system $R\subset B(H)$  such that $(R,\kappa)$ is an injective envelope of $S$, where $\kappa:R\to B(H)$ is the inclusion map.
\item For each $a\in \rC^*(S)$, there is an injective operator system $R\subset B(H)$  such that $a\in R$ and $(R,\kappa)$ is an injective envelope of $S$, where $\kappa:R\to B(H)$ is the inclusion map.
\end{enumerate}
\end{theorem}
\begin{proof}
(i) $\Rightarrow$ (ii): Let $(A,j)$ be a $\rC^*$-envelope of $S$. Plainly,  $(\rC^*(S),i)$ is a $\rC^*$-extension of $S$, and hence a $\rC^*$-support. By assumption, there is a unital, surjective, completely isometric map $\Gamma:A\to \rC^*(S)$ such that $\Gamma\circ j=\id$. By Proposition \ref{P:CEtrik}, we infer that $\Gamma$ is a $*$-isomorphism. Thus, $(\rC^*(S),i)$ is a  $\rC^*$-envelope of $S$.

(ii) $\Rightarrow$ (i):  Fix a $\rC^*$-support $X\subset B(H)$ of $S$. By Theorem \ref{T:Cstarext}, there is an $S$-projection $\psi:B(H)\to B(H)$ such that $X=\alpha_\psi^{-1}(\CE(S,\psi))$. Furthermore, denoting the range of $\psi$ by $R_\psi\subset B(H)$, we know that $\alpha_\psi\circ \psi:\rC^*(R_\psi)\to \CE(R_\psi,\psi)$ is a unital  $*$-homomorphism by construction of the Choi--Effros product. Let $\pi=(\alpha_\psi\circ \psi)|_{\rC^*(S)}$, which is a $*$-homomorphism satisfying $\pi|_S=\alpha_\psi|_S$. In particular, the image of $\pi$ is $\CE(S,\psi)$. 

Next, our assumption is that $(\rC^*(S),i)$ is a $\rC^*$-envelope of $S$, so there must exist a unital surjective $*$-homomorphism $\sigma:\CE(S,\psi)\to \rC^*(S)$ such that $\sigma \circ \alpha_\psi|_S=\id$. We have
\[
\sigma\circ \pi|_S=\sigma\circ \alpha_\psi|_S=\id 
\]
and
\[
\pi\circ \sigma\circ \alpha_\psi|_S=\pi|_S=\alpha_\psi|_S.
\]
Consequently, because $\pi$ and $\sigma$ are $*$-homomorphisms, we find $\pi=\sigma^{-1}$ so $\pi$ is completely isometric.
Hence, we may define $\Gamma=\alpha_\psi^{-1}\circ \pi:\rC^*(S)\to X$, which is a unital, surjective, completely isometric map satisfying $\Gamma|_S=\id$.  This establishes the desired statement for the $\rC^*$-envelope $(\rC^*(S),i)$. The corresponding statement for any $\rC^*$-envelope of $S$ follows readily by uniqueness.

(ii) $\Rightarrow$ (iii): We know that $(\rC^*(S),i)$ is a $\rC^*$-envelope of $S$. Thus, applying Corollary \ref{C:coversenv}, we find a minimal $S$-projection $\theta:B(H)\to B(H)$ such that  \[\rC^*(S)=\alpha_\theta ^{-1}(\CE(S,\theta)).\] In particular, this shows that $\rC^*(S)$ is contained in the range of $\theta$. Furthermore, if we let $\kappa:\ran\theta\to B(H)$ denote the inclusion, then $(\ran \theta, \kappa)$  is an injective envelope of $S$ by Theorem \ref{T:Hamana}.

(iii) $\Rightarrow$ (iv): This is trivial. 

(iv) $\Rightarrow$ (ii): Let $a\in \rC^*(S)$ be an element in the Shilov ideal of $S$ in $\rC^*(S)$. By assumption, there is an injective operator system $R\subset B(H)$  containing $a$ such that $(R,\kappa)$ is an injective envelope of $S$, where $\kappa:R\to B(H)$ denote the inclusion map. Since $R$ is injective, we may choose an $S$-projection $\theta:B(H)\to B(H)$ with range $R$. In turn, Theorem \ref{T:Hamana} implies that $\theta$ must be a minimal $S$-projection.  By virtue of Proposition \ref{P:CEcomp}(i) applied with $\psi=\id$, we see that $a$ must lie in the kernel of $\alpha_\theta\circ \theta$, and hence $\theta(a)=0$. On other hand, $\theta(a)=a$ since $a$ lies in the range of the idempotent $\theta$, so we conclude $a=0$. Therefore, the Shilov ideal of $S$ in $\rC^*(S)$ is trivial, so that $(\rC^*(S),i)$ is a  $\rC^*$-envelope of $S$.
\end{proof}

The previous result characterizes when the $\rC^*$-algebra $\rC^*(S)$ is contained in an injective envelope of $S$ inside of $B(H)$. The following complement addresses the question of when this $\rC^*$-algebra is contained in \emph{every} injective envelope of $S$ inside of $B(H)$.

\begin{corollary}\label{C:uniqueCE}
Let $S\subset B(H)$  be an operator system. Then, the following statements are equivalent.
\begin{enumerate}[{\rm (i)}]
\item $\rC^*(S)$ is the unique $\rC^*$-support of $S$.
\item $\rC^*(S)$ is contained  in every injective operator system $R\subset B(H)$ containing $S$.
\item $\rC^*(S)$ is contained  in every operator system $R\subset B(H)$ for which $(R,\kappa)$ is an injective envelope of $S$, where $\kappa:R\to B(H)$ is the inclusion map.
\end{enumerate}
\end{corollary}
\begin{proof}
(i) $\Rightarrow$ (ii):  Let  $R\subset B(H)$ be an injective operator system containing $S$. Let $\theta:B(H)\to B(H)$ be an $S$-projection with range $R$. Since $\alpha_\theta^{-1}(\CE(S,\theta))$ is a $\rC^*$-support of $S$, our assumption implies $\rC^*(S)=\alpha_\theta^{-1}(\CE(S,\theta))$, which is plainly contained in $R$.

(ii) $\Rightarrow$ (iii): This is trivial.

(iii) $\Rightarrow$ (i):  Let $X\subset B(H)$ be a $\rC^*$-support of $S$. By assumption, we see that $\rC^*(S)$ is contained  in every operator system $R\subset B(H)$ for which $(R,\kappa)$ is an injective envelope of $S$.  Invoking Theorem \ref{T:Shtrivial}, we infer that if $(A,j)$ is a $\rC^*$-envelope of $S$, then there is a unital, surjective, completely isometric map $\Gamma:A\to X$ such that $\Gamma\circ j=\id$. Hence, by Corollary \ref{C:coversenv} there is a minimal $S$-projection $\theta:B(H)\to B(H)$ such that  $X=\alpha_\theta ^{-1}(\CE(S,\theta)).$ Note also that $(\ran \theta,i)$ is an injective envelope of $S$ by Theorem \ref{T:Hamana}, where $i:\ran\theta\to B(H)$ is the inclusion. Applying our assumption again, we obtain that $\rC^*(S)$ is contained in the range of $\theta$. Finally, Proposition \ref{P:CEcomp}(ii) (with $\psi=\id$) implies $\rC^*(S)=\alpha_\theta^{-1}(\CE(S,\theta))=X$.
\end{proof}

\subsection{A new characterization of the unique extension property}\label{SS:uep}
Let $A$ be a  unital $\rC^*$-algebra and let $S\subset A$ be an operator system such that $A=\rC^*(S)$. Let $\pi:A\to B(H)$ be a unital $*$-representation. We say that $\pi$ has the \emph{unique extension property} with respect to $S$ if, given a unital completely positive map $\psi:A\to B(H)$ such that $\pi|_S=\psi|_S$, we must necessarily have that $\pi=\psi$.  

As an application of the main results of this section, we now give a new characterization of this property for injective representations.

\begin{corollary}\label{C:uep}
Let $A$ be a  unital $\rC^*$-algebra and let $S\subset A$ be an operator system such that $A=\rC^*(S)$. Let $\pi:A\to B(H)$ be an injective unital $*$-representation.  Then, the following statements are equivalent.
\begin{enumerate}[{\rm (i)}]
\item $\pi$ has the unique extension property with respect to $S$. 
\item $\pi(A)$ is the unique $\rC^*$-support of $\pi(S)$.
\end{enumerate}
\end{corollary}
\begin{proof}
Let $\F\subset B(H)$ denote the set of elements $t$ for which $\phi(t)=t$ for every $\pi(S)$-map $\phi:B(H)\to B(H)$. It is readily seen that the identity representation of $\pi(A)$ has the unique extension property with respect to $\pi(S)$ if and only if $\pi(A)\subset \F$. By virtue of Arveson's well-known invariance principle for the unique extension property (see for instance \cite[Theorem 2.1.2]{arveson1969}), we see that (i) is then equivalent to $\pi(A)\subset \F$.

Corollary \ref{C:uniqueCE} shows that statement (ii) is  equivalent to the fact that $\pi(A)\subset \bigcap R$, where the intersection is taken over all operators systems $R\subset B(H)$ such that $(R,i)$ is an injective envelope of $\pi(S)$ (here $i:R\to B(H)$ denotes the inclusion, as usual). Thus, the desired equivalence of (i) and (ii) will follow once we show that $\F=\bigcap R$, but this is precisely the content of  \cite[Proposition 4.3]{paulsen2011}.
\end{proof}

We remark that condition (ii) above can further be understood in terms of containment of $\pi(A)$ inside every injective operator system $R\subset B(H)$ that contains $\pi(S)$;  see Corollary \ref{C:uniqueCE}. 

The operator system $S$ is said to be \emph{hyperrigid} in $A$ if every unital $*$-representation of $A$ has the unique extension property with respect to $S$. The following is well known; it appears implicitly for instance the proof of \cite[Theorem 2.1]{arveson2011}. Since we lack an exact reference, we give the short argument here.

\begin{lemma}\label{L:hrisom}
Let $A$ be a  unital $\rC^*$-algebra and let $S\subset A$ be an operator system such that $A=\rC^*(S)$. Then, the following statements are equivalent.
\begin{enumerate}[{\rm (i)}]
\item $S$ is hyperrigid in $A$.
\item Every injective  unital $*$-representation  of $A$  has the unique extension property with respect to $S$.
\end{enumerate}
\end{lemma}
\begin{proof}
(i)$\Rightarrow$(ii): This is trivial.

(ii)$\Rightarrow$(i): Let $\pi:A\to B(H)$ be an arbitrary unital $*$-representation. Choose an injective unital $*$-representation $\sigma:A\to B(K)$. Then, $\pi\oplus \sigma$ is injective on $A$, so that $\pi\oplus \sigma$ has the unique extension property with respect to $S$ by assumption. In turn, by \cite[Lemma 2.8]{CTh2022}, we conclude that $\pi$ has the unique extension property with respect to $S$. Hence $S$ is hyperrigid in $A$.
\end{proof}

We can also give a new characterization of hyperrigidity for operator systems in terms of $\rC^*$-supports.

\begin{corollary}\label{C:hr}
Let $A$ be a  unital $\rC^*$-algebra and let $S\subset A$ be an operator system such that $A=\rC^*(S)$.  Then, the following statements are equivalent.
\begin{enumerate}[{\rm (i)}]
\item $S$ is hyperrigid in $A$.
\item Let $\pi:A\to B(H)$ be an injective unital $*$-representation. Then, $\pi(A)$ is the unique $\rC^*$-support of $\pi(S)$.
\end{enumerate}
\end{corollary}
\begin{proof}
This follows at once upon combining Corollary \ref{C:uep} with Lemma \ref{L:hrisom}.
\end{proof}

Hyperrigid operator systems are plentiful amongst naturally occurring concrete classes of operator systems; see for instance \cite{CH2018},\cite{kim2021},\cite{HK2019},\cite{KR2020},\cite{scherer2024} and the references therein. For such operator systems, the previous corollary shows that $\rC^*$-supports are severely constrained.

\section{The space of abnormalities}\label{S:ab}
As discussed in the introduction, one of our motivations for performing a detailed study of $\rC^*$-supports is to improve our understanding of abnormalities and to sharpen Kakariadis' result \eqref{Eq:Kak}. We tackle this question in this section.

Let $S\subset B(H)$ be an operator system and put $A=\rC^*(S)$. Let $E$ denote the set of unital completely positive maps $\phi:A\to B(H)$ agreeing with the identity on $S$. The set of  \emph{abnormalities} of $S$ is defined to be
\[
\Ab(\S)=\{\phi(a)-a:a\in A, \phi\in E\}.
\]
Our first aim is to show that the set of abnormalities is always contained in the kernel of certain maps. To streamline the exposition, we introduce the following terminology.
Let $R$ be another operator system and let $\pi:A\to R$ be a unital completely positive map. We say that $\pi$ has the \emph{unique $R$-extension property} with respect to $S$ if, given any unital completely positive map $\psi:A\to R$ such that $\pi|_S=\psi|_S$, it follows that $\pi=\psi$. The key point here is that we restrict the range of the extension $\psi$ to lie in $R$, thus making this notion a priori weaker than the previously introduced unique extension property (see Subsection \ref{SS:uep}). Such relaxed versions of the unique extension property have appeared in recent works \cite{CTh2024},\cite{DT2025}.

\begin{lemma}\label{L:abker}
Let $S\subset B(H)$ be an operator system and put $A=\rC^*(S)$. Then, the following statements hold.
\begin{enumerate}[{\rm(i)}]
\item Let $R$ be an operator system and let $\pi:B(H)\to R$ be a unital completely positive map with the property that $\pi|_{A}$ has the unique $R$-extension property with respect to $S$. Then, $\Ab(S)\subset \ker \pi$.
\item Let $\theta:B(H)\to B(H)$ be a minimal $S$-projection. Then, $\Ab(S)\subset \ker \theta$.
\end{enumerate}
\end{lemma}
\begin{proof}
(i) Let $d\in \Ab(S)$, so that there is a unital completely positive map $\phi:A\to B(H)$ with $\phi|_S=\id$  and $a\in A$ such that $d=\phi(a)-a$. We see that $\pi\circ \phi:A\to R$ is a unital completely positive extension of $\pi|_{S}$. Thus, the assumption on $\pi$ yields $\pi|_{A}=\pi\circ \phi$, so $\pi(a)=(\pi\circ \phi)(a)$. Equivalently, $\pi(d)=0$, as desired.

(ii) Let $R\subset B(H)$ denote the range of $\theta$. Then, $\theta|_{A}$ has the unique $R$-extension property with respect to $S$ by virtue of  Proposition \ref{P:minuep}, so the conclusion follows from (i).
\end{proof}

In what follows, we let $\Sigma\subset A$ denote the Shilov ideal of $S$ in $A$. Recall that if $(A,j)$ is a $\rC^*$-envelope of $S$,  $\Sigma$ is the kernel of the unique unital surjective $*$-homomorphism $\pi:\rC^*(S)\to A$ such that $\pi|_S=j$. Apply  Proposition \ref{P:CEcomp}(i) with $\psi=\id$ to see that 
\begin{equation}\label{Eq:shilov}
\Sigma=\ker (\alpha_\theta\circ \theta)\cap A
\end{equation}
for any minimal $S$-projection $\theta:B(H)\to B(H)$. Combining this with Lemma \ref{L:abker}, we see that 
\begin{equation}\label{Eq:Abincl}
\Ab(S)\cap A \subset \Sigma.
\end{equation}
 The converse inclusion also holds, as shown in \cite[Proposition 3.4]{kakariadis2013}. We record the statement for ease of reference below.

\begin{proposition}\label{P:abShi}
Let $S\subset B(H)$ be an operator system and let $A=\rC^*(S)$. Then, $\Sigma=\Ab(S)\cap A$.
\end{proposition}

Our next task is to give a description of the full set $\Ab(S)$.  We define the \emph{expanse} of $S$ to be the operator system 
\[
\Exp(S)=\spn\{\phi(A) \}
\] 
where the span is taken over all $S$-maps $\phi:B(H)\to B(H)$.  We now arrive at our announced description of the abnormalities, which is the second main result of the paper.

\begin{theorem}\label{T:abdescr}
Let $S\subset B(H)$ be an operator system. Then, \[\spn \Ab(S)=\Exp(S)\cap \ker \theta\] for every minimal $S$-projection $\theta:B(H)\to B(H)$.
\end{theorem}
\begin{proof}
Throughout the proof we write $A=\rC^*(S)$. It follows immediately from the definition that $\Ab(S)\subset \Exp(S)$. Fix a minimal $S$-projection $\theta:B(H)\to B(H)$.
Then,  \[\spn \Ab(S)\subset \Exp(S)\cap \ker \theta\]  by Lemma \ref{L:abker}. Conversely, let $t\in \Exp(S)\cap \ker \theta$. Then, there are elements $a_0,a_1,\ldots,a_n\in A$ and $S$-maps $\phi_1,\ldots,\phi_n$ such that $t=a_0+\sum_{j=1}^n \phi_j(a_j)$. Let $R\subset B(H)$ denote the range of $\theta$. For each $1\leq j\leq n$, we see that $\theta\circ \phi_j:B(H)\to R$ is a completely positive map with $\theta$ on $S$. By Proposition \ref{P:minuep}, we infer $\theta\circ \phi_j=\theta$ on $A$. On the other hand, we know that $t\in \ker \theta$, whence
\[
0=\theta(t)=\theta(a_0)+\sum_{j=1}^n (\theta\circ \phi_j)(a_j)=\theta( a_0+a_1+\ldots+a_n).
\]
Thus, the element $b=a_0+\ldots+a_n$ lies in $\ker \theta\cap A$, so in fact $b\in \Sigma$ by \eqref{Eq:shilov}. In turn, Proposition \ref{P:abShi} implies that  $b\in \Ab(S)$. Finally, we find
\[
t=a_0+\sum_{j=1}^n \phi_j(a_j) =b+\sum_{j=1}^n (\phi_j(a_j)-a_j)\in \spn \Ab(S).\qedhere
\]
\end{proof}

In light of the previous result, one may wonder if $\Ab(S)$ is a subspace of $B(H)$. It follows from Proposition \ref{P:abShi} that $\Ab(S)\cap A$ is indeed a subspace, but we do not know if the same is true of $\Ab(S)$.

Next,  Theorem \ref{T:abdescr} makes it desirable to gain a good understanding of the set $\Exp(S)$. Towards this goal, we examine a more structured subset of it.  We define the \emph{$\rC^*$-expanse} of $S$ as
\[
\rC^*\Exp(S)=\spn\{\theta(A) \}
\] 
where the span is taken over all minimal $S$-projections $\theta:B(H)\to B(H)$.

This last definition can be reformulated slightly. Indeed, let $\theta:B(H)\to B(H)$ be an $S$-projection. Because $\alpha_\theta\circ \theta$ is a $*$-homomorphism on $\rC^*(\ran \theta)$ and $A=\rC^*(S)$, it follows that $\CE(S,\theta)=(\alpha_\theta\circ \theta)(A)$, so that $\alpha_\theta^{-1}(\CE(S,\theta))=\theta(A)$. We conclude that
\begin{equation}\label{Eq:exp}
\rC^*\Exp(S)=\spn\{\alpha_\theta^{-1}(\CE(S,\theta)) \}
\end{equation}
where the span is taken over all minimal $S$-projections $\theta:B(H)\to B(H)$.
As a consequence, we obtain the following interpretation of the $\rC^*$-expanse.

\begin{theorem}\label{T:expanse}
Let $S\subset B(H)$ be an operator system and let $A=\rC^*(S)$. Assume that the Shilov ideal of $S$ in $A$ is trivial. Then, the following statements hold.
\begin{enumerate}[{\rm (i)}]
\item If $\phi:B(H)\to B(H)$ is an $S$-map, then there is a minimal $S$-projection $\theta:B(H)\to B(H)$ such that $\phi=\theta$ on $A$.
\item $\rC^*\Exp(S)=\Exp(S)$
\item $\rC^*\Exp(S)$ is the smallest subspace generated by all $\rC^*$-supports of $S$.
\end{enumerate}
\end{theorem}
\begin{proof}
(i) Since the Shilov ideal of $S$ in $A$ is trivial, we see that $(A,i)$ is a $\rC^*$-envelope of $S$, where $i:S\to B(H)$ is the inclusion map. We may apply Theorem \ref{T:Shtrivial} to find an injective operator system $R\subset B(H)$  containing $A$ such that $(R,j)$ is an injective envelope of $S$, where $j:R\to B(H)$ is the inclusion. By Theorem \ref{T:Hamana}, there is thus a minimal $S$-projection $\omega:B(H)\to B(H)$ with range $R$, and in particular  $\omega=\id$ on $A$.  Finally, by \cite[Proposition 2.11]{paulsen2011}, we see that the map $\theta=\phi\circ \omega$ is another minimal $S$-projection which satisfies $\phi=\theta$ on $A$.

(ii) Trivially, the inclusion $\rC^*\Exp(S\subset \Exp(S)$ always holds. Conversely,  we may apply (i) to see that $\Exp(S)\subset \rC^*\Exp(S)$.

(iii) Let $Y\subset B(H)$ denote the smallest subspace generated by all $\rC^*$-supports of $S$. By virtue of \eqref{Eq:exp}, we see that $ \rC^*\Exp(S)\subset Y$. Conversely,  let $X\subset B(H)$ be a $\rC^*$-support of $S$. Combining Corollary \ref{C:coversenv} with Theorem \ref{T:Shtrivial}, we see that there is a minimal $S$-projection $\theta:B(H)\to B(H)$ such that $X=\alpha_\theta^{-1}(\CE(S,\theta))\subset \rC^*\Exp(S)$. This shows that $Y\subset \rC^*\Exp(S)$.
\end{proof}

As an application, we show next that in the definition of $\Ab(S)$, it suffices to consider only those maps $\phi\in E$ that are minimal $S$-projections, provided that the Shilov ideal is assumed to be trivial. 

\begin{corollary}\label{C:abminproj}
Let $S\subset B(H)$ be an operator system and let $A=\rC^*(S)$. Assume that the Shilov ideal of $S$ in $A$ is trivial. Let $t\in B(H)$. Then, $t\in \Ab(S)$ if and only if there is a minimal $S$-projection $\theta:B(H)\to B(H)$ and an element $a\in A$ such that $t=\theta(a)-a$.
\end{corollary}
\begin{proof}
This follows immediately from Theorem \ref{T:expanse}(i).
\end{proof}

We close this section with an application of the previous results, motivated by the following complement to Proposition \ref{P:minuep}.

\begin{proposition}\label{P:quotienttight}
Let $S\subset B(H)$ be an operator system and let $A=\rC^*(S)$. Let $\theta:B(H)\to B(H)$ be a minimal $S$-projection, and let $q_\theta:B(H)\to B(H)/\ker\theta$ denote the quotient map. Then, $q_\theta$ is completely isometric on the range of $\theta$, and $q_\theta|_A$ has the  unique $(B(H)/\ker \theta)$-extension property with respect to $S$.
\end{proposition}
\begin{proof}
Throughout the proof, we let $R\subset B(H)$ denote the range of $\theta$.
As explained in \cite[Section 3]{KPTT2013}, $B(H)/\ker \theta$ admits a natural operator system structure with the property that $q_\theta$ is unital and completely positive.
Since $\theta$ is idempotent, we have that $\theta(t)-t\in \ker \theta$ for each $t\in B(H)$, so $q_\theta\circ \theta=q_\theta$.  This implies that the linear isomorphism
$\lambda_\theta:B(H)/\ker \theta\to R$ defined as
\[
\lambda_\theta(q_\theta(t))= \theta(t), \quad t\in B(H)
\]
satisfies $q_\theta\circ\lambda_\theta=\id$ on $B(H)/\ker\theta$. In particular, we see that $\lambda_\theta$ is completely isometric, and $q_\theta$ is completely isometric on $R$.

Let  $\phi:A\to B(H)/\ker \theta$ be a unital completely positive map agreeing with $q_\theta$ on $S$. Then, $\lambda_\theta\circ \phi:A\to R$ is a unital completely positive map agreeing with $\lambda_\theta\circ q_\theta=\theta$ on $S$. 
Hence, Proposition \ref{P:minuep} implies that $\lambda_\theta\circ \phi=\theta|_{A}$, so that 
\[
q_\theta|_{A}=(q_\theta\circ \theta)|_{A}=q_\theta\circ \lambda_\theta\circ \phi=\phi.
\]
\end{proof}

We now give an application of Theorem \ref{T:abdescr} that has a flavour similar to that of the previous proposition. Indeed, we show that, when the Shilov ideal is trivial, there is a single quotient of $B(H)$ that is completely isometric on $S$ and for which a conclusion reminiscent of Proposition \ref{P:quotienttight} can be obtained \emph{simultaneously} for every $S$-map on $A$.

For this purpose, we recall that given a closed subspace $D\subset B(H)$, the quotient $B(H)/D$ can be given a natural operator space structure \cite[Exercise 13.3]{paulsen2002}.

\begin{corollary}\label{C:uepquotient}
Let $S\subset B(H)$ be an operator system and let $A=\rC^*(S)$. Assume that the Shilov ideal of $S$ in $A$ is trivial. Let $D\subset B(H)$ denote the norm-closed subspace generated by $\Ab (S)$, and let $\delta:B(H)\to B(H)/D$ denote the quotient map. Then, the following statements hold for each unital completely positive map $\phi:A\to B(H)$ satisfying $\phi|_S=\id$.
\begin{enumerate}[{\rm (i)}]
\item The map $\delta$ is completely isometric on $\phi(A)$.
\item We have $\delta\circ \phi=\delta$ on $A$.
\end{enumerate}
\end{corollary}
\begin{proof}
(i)  Let $n\geq 1$ be an integer and let $[a_{ij}]\in \bM_n(A)$. Put $b=[\phi(a_{ij})]\in \bM_n(B(H))$. By Theorem  \ref{T:expanse}(i), we may find a minimal $S$-projection $\theta:B(H)\to B(H)$ such that $\phi=\theta$ on $A$. In particular, $\theta(\phi(a_{ij}))=\phi(a_{ij})$ for each $1\leq i,j\leq n$. Hence, $\theta^{(n)}(b)=b$. By Proposition \ref{P:quotienttight}, we see that $q_\theta$ is completely isometric on the range of $\theta$, so in particular
$
\|q_\theta^{(n)}(b)\|=\|b\|.
$
In turn,  it follows from Theorem \ref{T:abdescr} that $D\subset \ker \theta$, whence $\|\delta^{(n)}(b)\|\geq \|q_\theta^{(n)}(b)\|$ and $\|\delta^{(n)}(b)\|=\|b\|$ as desired.

(ii) For each $a\in A$, we see that $\phi(a)-a\in \Ab(S)\subset D$, whence $\delta(\phi(a)-a)=0$.
\end{proof}

\section{Some calculations of abnormalities}\label{S:example}
In this final section, we examine some examples from the perspective of our results on abnormalities.

\begin{example}\label{E:Ed}
Fix an integer $d\geq 2$. Let $\bF^2_d$ denote the full Fock space over $\bC^d$ and let $\lambda_1,\ldots,\lambda_d\in B(\bF^2_d)$ denote the isometric left creation operators. Let $V_d\subset B(\bF^2_d)$ denote the operator system generated by $\lambda_1,\ldots,\lambda_d$. Then, $V_d$ generates the Cuntz--Toeplitz $\rC^*$-algebra $\E_d\subset B(\bF^2_d)$, which contains the ideal $K$ of compact operators on $\bF^2_d$ \cite[Theorem 1.3]{popescu2006}.  The quotient $\E_d/K$ is $*$-isomorphic to the Cuntz algebra $\O_d$, and we let $q :\E_d\to \O_d$ denote the corresponding quotient map. For each $1\leq j\leq d$, we let $s_j=q(\lambda_j)$. Then, each $s_j$ is an isometry and $\sum_{j=1}^d s_j s_j^*=1$. By virtue of \cite[Theorem 3.1]{popescu1996}, we see that $q$ is completely isometric on $V_d$.   We claim that $\Ab(V_d)=K$. 

To see this, first note that since $\O_d$ is simple and $q$ is completely isometric on $V_d$, $(\O_d,q|_{V_d})$ is a $\rC^*$-envelope of $V_d$. Hence, the Shilov ideal of $V_d$ inside of $\E_d$ is simply $\ker q=K$. We conclude that $K=\Ab(V_d)\cap\E_d\subset \Ab(V_d)$ by Proposition \ref{P:abShi}.  For the converse, consider the quotient map $\pi:B(\bF^2_d)\to B(\bF^2_d)/K$. Choose a Hilbert space $H$ such that $B(\bF^2_d)/K\subset B(H)$, so that $\pi$ may be viewed as a unital completely positive map from $B(\bF^2_d)$ to $B(H)$ satisfying $\pi(\lambda_j)=s_j$ for each $1\leq j\leq d$. A well known multiplicative domain argument shows that $\pi$ has the unique extension property with respect to  $V_d$. Applying Lemma \ref{L:abker} we infer that $\Ab(V_d)\subset \ker \pi=K$.
\qed
\end{example}

Finally, we illustrate that the set $\Ab(S)\cap \rC^*(S)$ may be much smaller than $\Ab(S)$, thereby showing that our Theorem \ref{T:abdescr} gives information that was previously unattainable via Kakariadis' original result \eqref{Eq:Kak}.

\begin{example}\label{E:BDO}
Let $A_0$ be a unital $\rC^*$-algebra generated by some operator system $S_0$. Assume that all irreducible $*$-representations of $A_0$ have the unique extension property with respect to $S_0$, while $S_0$ is \emph{not} hyperrigid in $A_0$.  In other words, we are assuming that the inclusion $S_0\subset A_0$ fails to satisfy Arveson's hyperrigidity conjecture \cite{arveson2011}; such operator systems do exist as recently shown in \cite{BDO2024}. 

By Lemma \ref{L:hrisom}, there is an injective unital $*$-representation $\pi:A_0\to B(H)$ without the unique extension property with respect to $S_0$.  Let $S=\pi(S_0)$ and $A=\pi(A_0)$, so that $A=\rC^*(S)$. By \cite[Theorem 2.1.2]{arveson1969}  it then follows that the identity representation of $A$ does not have the unique extension property with respect to $S$, so $\Ab(S)$ is non-zero. 

On the other hand, by our assumption on $S_0$ and $A_0$, it is still true that every irreducible $*$-representation of $A$ has the unique extension property with respect to $S$. In particular, 
the Shilov ideal of $S$ in $A$ is necessarily trivial \cite[Theorem 2.2.3]{arveson1969}, so that $\Ab(S)\cap A=\{0\}$ by Proposition \ref{P:abShi}.  Thus, in this case we see that the inclusion 
\[
\{0\}=\Ab(S)\cap A\subset \Ab(S)
\]
is strict. 

By definition of $\Ab(S)$, this is equivalent to the existence of an $S$-map $\phi:B(H)\to B(H)$ that does not send $A$ back into itself.  We note in passing that this conclusion is consistent with the findings of \cite[Corollary 4.5]{thompson2024}, since in this case $\pi$ has the so-called approximate unique extension property with respect to $S_0$. \qed
\end{example}

\bibliographystyle{plain}
\bibliography{abnormalities}

\def\cprime{$'$}
\begin{bibdiv}
\begin{biblist}

\bib{arveson1969}{article}{
      author={Arveson, William},
       title={Subalgebras of {$C^{\ast} $}-algebras},
        date={1969},
        ISSN={0001-5962},
     journal={Acta Math.},
      volume={123},
       pages={141\ndash 224},
      review={\MR{0253059 (40 \#6274)}},
}

\bib{arveson1972}{article}{
      author={Arveson, William},
       title={Subalgebras of {$C^{\ast} $}-algebras. {II}},
        date={1972},
        ISSN={0001-5962},
     journal={Acta Math.},
      volume={128},
      number={3-4},
       pages={271\ndash 308},
      review={\MR{0394232 (52 \#15035)}},
}

\bib{arveson2008}{article}{
      author={Arveson, William},
       title={The noncommutative {C}hoquet boundary},
        date={2008},
        ISSN={0894-0347},
     journal={J. Amer. Math. Soc.},
      volume={21},
      number={4},
       pages={1065\ndash 1084},
         url={http://dx.doi.org/10.1090/S0894-0347-07-00570-X},
      review={\MR{2425180 (2009g:46108)}},
}

\bib{arveson2011}{article}{
      author={Arveson, William},
       title={The noncommutative {C}hoquet boundary {II}: hyperrigidity},
        date={2011},
        ISSN={0021-2172},
     journal={Israel J. Math.},
      volume={184},
       pages={349\ndash 385},
         url={http://dx.doi.org/10.1007/s11856-011-0071-z},
      review={\MR{2823981}},
}

\bib{bilich2025}{article}{
      author={Bilich, Boris},
       title={Maximality of correspondence representations},
        date={2025},
        ISSN={0001-8708,1090-2082},
     journal={Adv. Math.},
      volume={462},
       pages={Paper No. 110097},
         url={https://doi.org/10.1016/j.aim.2024.110097},
      review={\MR{4846033}},
}

\bib{BDO2024}{article}{
      author={Bilich, Boris},
      author={Dor-On, Adam},
       title={Arveson's hyperrigidity conjecture is false},
        date={2024},
     journal={arXiv preprint arXiv:2404.05018},
}

\bib{CE1976}{article}{
      author={Choi, Man~Duen},
      author={Effros, Edward~G.},
       title={The completely positive lifting problem for {$C\sp*$}-algebras},
        date={1976},
        ISSN={0003-486X},
     journal={Ann. of Math. (2)},
      volume={104},
      number={3},
       pages={585\ndash 609},
         url={https://doi.org/10.2307/1970968},
      review={\MR{417795}},
}

\bib{CE1977}{article}{
      author={Choi, Man~Duen},
      author={Effros, Edward~G.},
       title={Injectivity and operator spaces},
        date={1977},
        ISSN={0022-1236},
     journal={J. Functional Analysis},
      volume={24},
      number={2},
       pages={156\ndash 209},
         url={https://doi.org/10.1016/0022-1236(77)90052-0},
      review={\MR{430809}},
}

\bib{CH2018}{article}{
      author={Clou\^atre, Rapha\"el},
      author={Hartz, Michael},
       title={Multiplier algebras of complete {N}evanlinna-{P}ick spaces:
  dilations, boundary representations and hyperrigidity},
        date={2018},
        ISSN={0022-1236},
     journal={J. Funct. Anal.},
      volume={274},
      number={6},
       pages={1690\ndash 1738},
         url={https://doi-org.uml.idm.oclc.org/10.1016/j.jfa.2017.10.008},
      review={\MR{3758546}},
}

\bib{CS2023}{article}{
      author={Clou{\^a}tre, Rapha{\"e}l},
      author={Saikia, Hridoyananda},
       title={A boundary projection for the dilation order},
        date={2023},
     journal={arXiv preprint arXiv:2310.17601},
}

\bib{CTh2022}{article}{
      author={Clou\^{a}tre, Rapha\"{e}l},
      author={Thompson, Ian},
       title={Finite dimensionality in the non-commutative {C}hoquet boundary:
  peaking phenomena and {$\rm C^*$}-liminality},
        date={2022},
        ISSN={1073-7928},
     journal={Int. Math. Res. Not. IMRN},
      number={20},
       pages={16046\ndash 16093},
         url={https://doi-org.uml.idm.oclc.org/10.1093/imrn/rnab184},
      review={\MR{4498170}},
}

\bib{CTh2023}{article}{
      author={Clou\^atre, Rapha\"el},
      author={Thompson, Ian},
       title={Minimal boundaries for operator algebras},
        date={2023},
        ISSN={2330-0000},
     journal={Trans. Amer. Math. Soc. Ser. B},
      volume={10},
       pages={807\ndash 832},
         url={https://doi.org/10.1090/btran/154},
      review={\MR{4611189}},
}

\bib{CTh2024}{article}{
      author={Clou\^atre, Rapha\"el},
      author={Thompson, Ian},
       title={Rigidity of operator systems: tight extensions and noncommutative
  measurable structures},
        date={2024},
     journal={arXiv preprint arXiv:2406.16806},
}

\bib{courtney2025}{article}{
      author={Courtney, Kristin},
       title={Completely positive approximations and inductive systems},
        date={2025},
     journal={arXiv:2304.02325},
}

\bib{CW2025}{article}{
      author={Courtney, Kristin},
      author={Winter, Wilhelm},
       title={Nuclearity and {${\rm CPC}^*$}-systems},
        date={2025},
        ISSN={2050-5094},
     journal={Forum Math. Sigma},
      volume={13},
       pages={Paper No. e59, 24},
         url={https://doi.org/10.1017/fms.2024.123},
      review={\MR{4880920}},
}

\bib{DK2015}{article}{
      author={Davidson, Kenneth~R.},
      author={Kennedy, Matthew},
       title={The {C}hoquet boundary of an operator system},
        date={2015},
        ISSN={0012-7094},
     journal={Duke Math. J.},
      volume={164},
      number={15},
       pages={2989\ndash 3004},
         url={http://dx.doi.org.uml.idm.oclc.org/10.1215/00127094-3165004},
      review={\MR{3430455}},
}

\bib{DK2019}{article}{
      author={Davidson, Kenneth~R.},
      author={Kennedy, Matthew},
       title={Noncommutative {C}hoquet theory},
        date={2019},
     journal={preprint arXiv:1905.08436},
}

\bib{DK2021}{article}{
      author={Davidson, Kenneth~R},
      author={Kennedy, Matthew},
       title={Choquet order and hyperrigidity for function systems},
        date={2021},
     journal={Advances in Mathematics},
      volume={385},
       pages={107774},
}

\bib{DT2025}{article}{
      author={Dor-On, Adam},
      author={Thompson, Ian},
       title={The {H}ao-{N}g isomorphism theorem for reduced crossed products},
        date={2025},
     journal={arXiv:2505.00587},
}

\bib{dritschel2005}{article}{
      author={Dritschel, Michael~A.},
      author={McCullough, Scott~A.},
       title={Boundary representations for families of representations of
  operator algebras and spaces},
        date={2005},
        ISSN={0379-4024},
     journal={J. Operator Theory},
      volume={53},
      number={1},
       pages={159\ndash 167},
      review={\MR{2132691 (2006a:47095)}},
}

\bib{ER2022}{book}{
      author={Effros, Edward~G.},
      author={Ruan, Zhong-Jin},
       title={Theory of operator spaces},
   publisher={AMS Chelsea Publishing, Providence, RI},
        date={2022},
        ISBN={978-1-4704-6505-6},
         url={https://doi.org/10.1090/chel/386},
        note={Corrected reprint of the 2000 original [1793753]},
      review={\MR{4420075}},
}

\bib{EH2019}{article}{
      author={Evert, Eric},
      author={Helton, J.~William},
       title={Arveson extreme points span free spectrahedra},
        date={2019},
        ISSN={0025-5831,1432-1807},
     journal={Math. Ann.},
      volume={375},
      number={1-2},
       pages={629\ndash 653},
         url={https://doi.org/10.1007/s00208-019-01858-9},
      review={\MR{4000252}},
}

\bib{EHKM2018}{article}{
      author={Evert, Eric},
      author={Helton, J.~William},
      author={Klep, Igor},
      author={McCullough, Scott},
       title={Extreme points of matrix convex sets, free spectrahedra, and
  dilation theory},
        date={2018},
        ISSN={1050-6926},
     journal={J. Geom. Anal.},
      volume={28},
      number={2},
       pages={1373\ndash 1408},
         url={https://doi-org.uml.idm.oclc.org/10.1007/s12220-017-9866-4},
      review={\MR{3790504}},
}

\bib{hamana1979}{article}{
      author={Hamana, Masamichi},
       title={Injective envelopes of operator systems},
        date={1979},
        ISSN={0034-5318},
     journal={Publ. Res. Inst. Math. Sci.},
      volume={15},
      number={3},
       pages={773\ndash 785},
  url={http://dx.doi.org.proxy.lib.uwaterloo.ca/10.2977/prims/1195187876},
      review={\MR{566081 (81h:46071)}},
}

\bib{HK2019}{article}{
      author={Harris, Samuel~J},
      author={Kim, Se-Jin},
       title={Crossed products of operator systems},
        date={2019},
     journal={Journal of Functional Analysis},
      volume={276},
      number={7},
       pages={2156\ndash 2193},
}

\bib{kakariadis2013}{article}{
      author={Kakariadis, Evgenios T.~A.},
       title={The { \v S}ilov boundary for operator spaces},
        date={2013},
        ISSN={0378-620X,1420-8989},
     journal={Integral Equations Operator Theory},
      volume={76},
      number={1},
       pages={25\ndash 38},
         url={https://doi.org/10.1007/s00020-013-2033-9},
      review={\MR{3041719}},
}

\bib{KR2020}{article}{
      author={Katsoulis, Elias~G.},
      author={Ramsey, Christopher},
       title={The hyperrigidity of tensor algebras of {$\rm
  C^\ast$}-correspondences},
        date={2020},
        ISSN={0022-247X,1096-0813},
     journal={J. Math. Anal. Appl.},
      volume={483},
      number={1},
       pages={123611, 10},
         url={https://doi.org/10.1016/j.jmaa.2019.123611},
      review={\MR{4023882}},
}

\bib{KPTT2013}{article}{
      author={Kavruk, Ali~S.},
      author={Paulsen, Vern~I.},
      author={Todorov, Ivan~G.},
      author={Tomforde, Mark},
       title={Quotients, exactness, and nuclearity in the operator system
  category},
        date={2013},
        ISSN={0001-8708,1090-2082},
     journal={Adv. Math.},
      volume={235},
       pages={321\ndash 360},
         url={https://doi.org/10.1016/j.aim.2012.05.025},
      review={\MR{3010061}},
}

\bib{kim2021}{article}{
      author={Kim, Se-Jin},
       title={Hyperrigidity of {C}*-correspondences},
        date={2021},
        ISSN={0378-620X,1420-8989},
     journal={Integral Equations Operator Theory},
      volume={93},
      number={4},
       pages={Paper No. 47, 17},
         url={https://doi.org/10.1007/s00020-021-02663-3},
      review={\MR{4292770}},
}

\bib{KW1998}{article}{
      author={Kirchberg, Eberhard},
      author={Wassermann, Simon},
       title={{$C^\ast$}-algebras generated by operator systems},
        date={1998},
        ISSN={0022-1236,1096-0783},
     journal={J. Funct. Anal.},
      volume={155},
      number={2},
       pages={324\ndash 351},
         url={https://doi.org/10.1006/jfan.1997.3226},
      review={\MR{1624549}},
}

\bib{kleski2014bdry}{article}{
      author={Kleski, Craig},
       title={Boundary representations and pure completely positive maps},
        date={2014},
        ISSN={0379-4024},
     journal={J. Operator Theory},
      volume={71},
      number={1},
       pages={45\ndash 62},
         url={http://dx.doi.org.uml.idm.oclc.org/10.7900/jot.2011oct22.1927},
      review={\MR{3173052}},
}

\bib{paulsen2002}{book}{
      author={Paulsen, Vern},
       title={Completely bounded maps and operator algebras},
      series={Cambridge Studies in Advanced Mathematics},
   publisher={Cambridge University Press},
     address={Cambridge},
        date={2002},
      volume={78},
        ISBN={0-521-81669-6},
      review={\MR{1976867 (2004c:46118)}},
}

\bib{paulsen2011}{article}{
      author={Paulsen, Vern~I.},
       title={Weak expectations and the injective envelope},
        date={2011},
        ISSN={0002-9947,1088-6850},
     journal={Trans. Amer. Math. Soc.},
      volume={363},
      number={9},
       pages={4735\ndash 4755},
         url={https://doi.org/10.1090/S0002-9947-2011-05203-7},
      review={\MR{2806689}},
}

\bib{popescu1996}{article}{
      author={Popescu, Gelu},
       title={Non-commutative disc algebras and their representations},
        date={1996},
        ISSN={0002-9939},
     journal={Proc. Amer. Math. Soc.},
      volume={124},
      number={7},
       pages={2137\ndash 2148},
  url={http://dx.doi.org.proxy.lib.uwaterloo.ca/10.1090/S0002-9939-96-03514-9},
      review={\MR{1343719 (96k:47077)}},
}

\bib{popescu2006}{article}{
      author={Popescu, Gelu},
       title={Operator theory on noncommutative varieties},
        date={2006},
        ISSN={0022-2518},
     journal={Indiana Univ. Math. J.},
      volume={55},
      number={2},
       pages={389\ndash 442},
  url={http://dx.doi.org.proxy.lib.uwaterloo.ca/10.1512/iumj.2006.55.2771},
      review={\MR{2225440 (2007m:47008)}},
}

\bib{scherer2024}{article}{
      author={Scherer, Marcel},
       title={The hyperrigidity conjecture for compact convex sets in
  $\mathbb{R}^2$},
        date={2024},
     journal={arXiv:2411.11709},
}

\bib{thompson2024}{article}{
      author={Thompson, Ian},
       title={An approximate unique extension property for completely positive
  maps},
        date={2024},
        ISSN={0022-1236,1096-0783},
     journal={J. Funct. Anal.},
      volume={286},
      number={1},
       pages={Paper No. 110193, 32},
         url={https://doi.org/10.1016/j.jfa.2023.110193},
      review={\MR{4654017}},
}

\end{biblist}
\end{bibdiv}
	
\end{document}